\documentclass[a4paper,11pt,reqno]{amsart}

\usepackage[headings]{fullpage}

\usepackage{mathtools}
\usepackage{amssymb}
\usepackage{tensor}

\usepackage{tikz}
\usetikzlibrary{cd}

\usepackage[backref]{hyperref}

\usepackage[nobysame,alphabetic,initials,msc-links]{amsrefs}

\DefineSimpleKey{bib}{how}

\renewcommand{\eprint}[1]{#1}
\BibSpec{misc}{%
  +{}{\PrintAuthors}  {author}
  +{,}{ \textit}      {title}
  +{,}{ }             {how}
  +{}{ \parenthesize} {date}
  +{,} { available at \eprint}        {eprint}
  +{,}{ available at \url}{url}
  +{,}{ }             {note}
  +{.}{}              {transition}
}

\mathchardef\mhyph="2D
\DeclareSymbolFont{stmry}{U}{stmry}{m}{n}
\DeclareMathSymbol\stmryolt\mathbin{stmry}{"3C}
\DeclareMathSymbol\stmryogt\mathbin{stmry}{"3D}

\newcommand*\lon{%
        \nobreak
        \mskip6mu plus1mu
        \mathpunct{}%
        \nonscript
        \mkern-\thinmuskip
        {:}%
        \mskip2mu
        \relax
}

\numberwithin{equation}{section}

\newtheorem{theorem}{Theorem}[section]

\newtheorem{lemma}[theorem]{Lemma}
\newtheorem{proposition}[theorem]{Proposition}

\theoremstyle{remark}
\newtheorem{remark}[theorem]{Remark}
\newtheorem{example}[theorem]{Example}

\theoremstyle{definition}
\newtheorem{definition}[theorem]{Definition}

\newcommand{\C}{\mathbb{C}}

\newcommand{\cA}{\mathcal{A}}
\newcommand{\cB}{\mathcal{B}}
\newcommand{\cC}{\mathcal{C}}
\newcommand{\cD}{\mathcal{D}}

\newcommand{\cK}{\mathcal{K}}

\newcommand{\cM}{\mathcal{M}}

\newcommand{\cO}{\mathcal{O}}

\newcommand{\cR}{\mathcal{R}}

\newcommand{\cZ}{\mathcal{Z}}

\DeclareMathOperator{\End}{End}
\DeclareMathOperator{\Hom}{Hom}

\DeclareMathOperator{\Irr}{Irr}

\DeclareMathOperator{\Mor}{Mor}

\DeclareMathOperator{\Nat}{Nat}

\DeclareMathOperator{\Rep}{Rep}

\DeclareMathOperator{\Id}{Id}
\DeclareMathOperator{\id}{id}

\newcommand{\norm}[1]{\left \| #1 \right \|}

\newcommand{\vnotimes}{\mathbin{\bar{\otimes}}}
\newcommand{\tilderhd}{\mathbin{\tilde{\rhd}}}

\newcommand{\Hilb}{\mathrm{Hilb}}

\newcommand{\CB}{\mathcal{CB}}
\newcommand{\cCB}{\mathcal{CB}^c}

\newcommand{\YD}{\mathcal{YD}}

\title[Yetter--Drinfeld C$^*$-algebras]{Categorical dualtiy for Yetter--Drinfeld C$^*$-algebras\\
Beyond the braided-commutative case}

\date{April 23, 2025}

\author{Lucas Hataishi}
\address{University of Oxford}
\email{lucas.yudihataishi@maths.ox.ac.uk}

\author{Makoto Yamashita}
\address{Universitetet i Oslo}
\email{makotoy@math.uio.no}

\thanks{Supported by the NFR project 300837 ``Quantum Symmetry''.}

\begin{document}

\begin{abstract}
  We develop a tensor categorical duality in the sprit of the Tannaka--Krein duality for the C$^*$-algebras admitting the Yetter--Drinfeld module structure over a compact quantum group.
  Under this duality, given a reduced compact quantum group $G$, the Yetter--Drinfeld $G$-C$^*$-algebras correspond to the bimodule categories over the representation category $\Rep(G)$, satisfying a certain centrality condition.
\end{abstract}

\maketitle

\section{Introduction}

The Tannaka--Krein duality principle for quantum groups has been an effective tool to compare different incarnations of \emph{quantum symmetries}.
In the original form by Woronowicz~\cite{MR943923}, it states that compact quantum groups are essentially the same as a pair of a rigid C$^*$-tensor category and a tensor functor into the category of finite dimensional Hilbert spaces.
More recently, this was generalized to actions of quantum groups, which are modeled by comodule algebras on the algebraic side, and module categories on the tensor categorical side.

In the operator algebraic setting, Tannaka--Krein theory is mostly developed for compact quantum groups $G$, as the complete reducibility of the representations of $G$ allows us to control possibly infinite-dimensional algebras with an action of $G$ in an algebraic way.
While a straightforward generalization of this to locally compact quantum groups seems hopeless, there is one promising class of locally compact quantum groups with compact-by-discrete type structure like totally disconnected groups: the Drinfeld double quantum group $DG$, which can be considered as a matched pair of a compact quantum group $G$ and its discrete dual quantum group $\hat{G}$.

Actions of $DG$ on C$^*$-algebras can be described as actions of $G$ with additional structure encoding the action of $\hat{G}$.
Since $G$-actions can be described categorically, one might hope to describe $DG$-actions categorically as well, and obtain a Tannaka--Krein duality for $DG$-actions.
This program was initiated in~\cite{NY3}, where Neshveyev and the second named author found that embeddings of the C$^*$-tensor category $\Rep(G)$ of finite dimensional unitary representations of $G$ into other C$^*$-tensor categories correspond to continuous unital $DG$-C$^*$-algebras having \emph{braided-commutativity}.
In this paper, we remove this braided commutativity assumption, obtaining a categorical description of general continuous unital Yetter--Drinfeld C$^*$-algebras.
Specifically, we establish a categorical equivalence between the category of such algebras and the category of certain bimodule categories over the representation category $\Rep(G)$.
A similar duality statement was recently given in~\cite{FreslonTaipeWang} at the level weak monoidal functors developing the line of works by Pinzari--Roberts~\cite{pr1} and Neshveyev~\cite{n1}.

In a follow-up paper~\cite{HY2022}, we will use the results of this paper to show that the theory of boundary actions for Drinfeld doubles of compact quantum groups is invariant under monoidal equivalence.
In particular, for two monoidally equivalent compact quantum groups $G$ and $G'$, the Furstenberg-Hamana boundary $\partial_{FH}(DG)$ of $DG$ is trivial if and only if $\partial_{FH}(DG')$ is trivial.
This answers a question left open in~\cite{HHN1}.

Another potential application of the results in this paper is in geometric representation theory for quantum groups.
There is a non-trivial relation between the topology of surfaces and module categories over braided tensor categories~\cite{BZBJ2018}.
The category of unitary representations of the Drinfeld double $DG$ of a compact quantum group $G$ is the \emph{unitary Drinfeld center} $\cZ(\Hilb(\Rep(G)))$, where $\Hilb(\Rep(G))$ is the category of all unitary representations of $G$, i.e., it allows for infinite dimensional representations, see~\cite{NY2}.

\medskip
Now let us describe the content of the paper.

In Section~\ref{sec:preliminaries}, we collect the preliminary materials on rigid C$^*$-tensor categories, compact quantum groups, and their Drinfeld doubles.

In Section~\ref{sec:centr-br-bimod-cat}, we introduce what will later be shown to be the Tannaka--Krein categorical dual of  continuous unital Yetter--Drinfeld C$^*$-algebras.
In order to account for the extra structures and properties of Yetter--Drinfeld C$^*$-algebras, we consider bimodule categories $(\cM,m)$ over $\Rep(G)$ with a distinguished $\Rep(G)$-generator $m$, equipped with unitary ismorphisms
\[
  \sigma_U \colon U \stmryogt m \to m \stmryolt U, \quad (U \in \Rep(G)),
\]
which are natural in $U$ and satisfy a \emph{half-braided} condition:
\[
  \sigma_{U \otimes V} = (\sigma_U \stmryolt \id_V) \circ (\id_U \stmryogt \sigma_V).
\]
We call such categories \emph{centrally pointed cyclic} bimodule C$^*$-$\Rep(G)$-module categories.

We prove our main result in Section~\ref{sec:duality-for-YD-algs}, giving the equivalence between the category of continuous unital Yetter--Drinfeld C$^*$-algebras and the category of centrally pointed cyclic C$^*$-$\Rep(G)$-bimodule categories.
It is based on the following observation.
Suppose $B$ is a Yetter--Drinfeld C$^*$-algebra.
The extra categorical structure we are looking for boils down to having a $G$-equivariant $B$-bimodule structure on the $G$-equivariant right $B$-module $H_U \otimes B$, where $H_U$ is the underlying finite-dimensional Hilbert space for an object $U$ in $\Rep(G)$.
Let us write the left action $\beta\colon \cO(G) \odot B \to B$ of the regular subalgebra $\cO(G) \subset C(G)$ as $\beta(x \odot b) = x \triangleright b$, and write $U$ as
\[
  U = \sum_{ij} m^U_{ij} \otimes u_{ij} \in \cB(H_U) \otimes C(G).
\]
We can define a left $B$-action $\pi_U\colon B \to \End_B(H_U \otimes B)$ on $H_U \otimes B$ by
\[
  \pi_U(b) = \sum_{ij} m^U_{ij} \otimes (u_{ij} \triangleright b).
\]
In~\cite{NY3}, the braided-commutativity assumption was used to give an alternative description of this action, allowing the proof that it indeed defines a $G$-equivariant $*$-homomorphism.
In Proposition~\ref{prop:equivariancecanonicalleftaction}, we show that braided-commutativity is not needed, proving that $\pi_U$ is $G$-equivariant using only the relation between the coaction of $\cO(G)$ and $\beta$ fitting in a Yetter--Drinfeld module.

\bigskip
\paragraph{Acknowledgements}
We would like to thank Sergey Neshveyev for discussions during the development of this work, and the anonymous reviewers for their helpful comments which helped us improve the presentation of the paper. For the purpose of Open Access, the authors have applied a CC BY public copyright licence to any Author Accepted Manuscript (AAM) version arising from this submission.

\section{Preliminaries}\label{sec:preliminaries}

\subsection{\texorpdfstring{C$^*$}{C*}-tensor categories}

Here we mostly follow~\citelist{\cite{MR3204665}\cite{NY3}}.
Given a \emph{C$^*$-category} $\cC$, we denote the space of morphisms from an object $X$ to another $Y$ as $\cC(X, Y)$.
The involution $\cC(X, Y) \to \cC(Y, X)$ is denoted by $T \mapsto T^*$, and the norm is by $\norm{T}$, so that we have the C$^*$-identity $\norm{T^* T} = \norm{T}^2$.
We also assume that $\cC$ has finite direct sums.
Given $X,Y \in \cC$, their direct sum is denoted by $X \oplus Y$ and there are isometries $\iota_X\colon X \to X \oplus Y \leftarrow Y \lon \iota_Y$ such that $\iota_X \circ \iota_X^* + \iota_Y \circ \iota_Y^* = \id_{X \oplus Y}$.
Under these assumptions, $\cC(X, Y)$ is naturally a right Hilbert module over the unital C$^*$-algebra $\cC(X) = \cC(X, X)$.
We tacitly assume that $\cC$ is closed under taking subobjects, i.e., any projection in $\cC(X)$ corresponds to a direct summand of $X$.

A \emph{C$^*$-tensor category} is a C$^*$-category endowed with monoidal structure given by a $*$-bifunctor $\otimes\colon \cC \times \cC \to \cC$, a unit object $1_\cC$, and unitary natural isomorphisms
\begin{equation*}
  1_\cC \otimes U \to U \leftarrow U \otimes 1_\cC, \qquad \Phi\colon (U \otimes V) \otimes W \to U \otimes (V \otimes W)
\end{equation*}
for $U, V, W \in \cC$, satisfying a standard set of axioms.
Without losing generality, we may and do assume that $\cC$ is \emph{strict} so that the above morphisms are identity, and that $1_\cC$ is simple, unless explicitly stated otherwise.

A \emph{rigid} C$^*$-tensor category is a C$^*$-tensor category where any object $U$ has a dual given by: an object $\bar U$ and morphisms
\begin{equation*}
  R \colon 1_\cC \to \bar U \otimes U, \qquad \bar R \colon 1_\cC \to U \otimes \bar U
\end{equation*}
satisfying the conjugate equations for $U$.

When $\cC$ is a C$^*$-tensor category, a \emph{right (C$^*$-)$\cC$-module category} is given by a C$^*$-category $\cM$, together with a $*$-bifunctor $\stmryolt \colon \cM \times \cC \to \cM$ and unitary natural isomorphisms
\begin{equation*}
  X \stmryolt 1_\cC  \to X, \qquad \Psi\colon (X \stmryolt U) \stmryolt V \to X \stmryolt (U \otimes V)
\end{equation*}
for $X \in \cM$ and $U, V \in \cC$, satisfying standard set of axioms.
Again we may and do assume that module category structures are strict so that the above morphisms are identities.

A \emph{functor of right $\cC$-module categories} is given by a functor $F\colon \cM \to \cM'$ of the underlying linear categories, together with natural isomorphisms
\[
  F_2 = F_{2;m,U}\colon F(m) \stmryolt U \to F(m \stmryolt U) \quad (m \in \cM, U \in \cC)
\]
satisfying the standard compatibility conditions with structure morphisms of $\cM$ and $\cM'$.
If $\cM$ and $\cM'$ are C$^*$-$\cC$-module categories, a functor $(F,F_2)$ as above is a said to be a \emph{functor of right C$^*$-$\cC$-module categories} if it is a $*$-functor and the natural isomorphism $F_2$ is unitary.

\subsection{Quantum groups}

Again we mostly follow~\citelist{\cite{MR3204665}\cite{NY3}}.
A \emph{compact quantum group} $G$ is given by a unital C$^*$-algebra $C(G)$ and a $*$-homomorphism $\Delta\colon C(G) \to C(G) \otimes C(G)$ satisfying the coassociativity and cancellation.
We denote its invariant state (\emph{Haar state}) by $h$, and always assume that $h$ is faithful on $C(G)$.
The GNS Hilbert space of $(C(G), h)$ is denoted by $L^2(G)$.

A \emph{finite dimensional unitary representation} of $G$ is given by a finite dimensional Hilbert space $H$ and a unitary element $U \in \cB(H) \otimes C(G)$ satisfying
\begin{equation}\label{eq:fin-dim-unitar-rep}
  U_{1 2} U_{1 3} = (\id \otimes \Delta)(U).
\end{equation}
We often write $H = H_U$.
These objects form a rigid C$^*$-tensor category that we denote by $\Rep(G)$, with its tensor unit represented by the multiplicative unit $1 \in C(G) \simeq \cB(\C) \otimes C(G)$.
Our convention of duality is given by $H_{\bar U} = \overline{H_U}$, with structure morphisms $R_U\colon 1 \to \bar U \otimes U$ and $\bar R_U \colon 1 \to U \otimes \bar U$ given by
\begin{align}\label{eq:duality-mors}
  R_U & = \sum_i \bar \xi_i \otimes \rho_U^{-1/2} \xi_i, & \bar R_U & = \sum_i \rho_U^{1/2} \xi_i \otimes \bar \xi_i
\end{align}
for any choice of orthonormal basis $(\xi_i)_i \subset H_U$.
Here $\rho_U$ is the positive operator on $H_U$ defined by
\[
  \rho_U = (\id \otimes f_1)(U),
\]
where $f_1$ is a special value of the Woronowicz characters $(f_z)_{z \in \C}$.

We fix an index set for the set of isomorphism classes of irreducible representations of $G$, and denote it by $\Irr(G)$.
Let $(U_i)_{i \in \Irr(G)}$ by be a choice of representatives of finite dimensional irreducible unitary representations, with underlying Hilbert spaces $(H_i)_i$.
For each $i \in \Irr(G)$ we choose a complete set of matrix units $(m^{(i)}_{kl})_{k,l}$ in $\cB(H_i)$ and write
\[
  U_i = \sum_{k,l} u^{(i)}_{kl} \otimes m^{(i)}_{kl}.
\]

As the model of $c_0$-sequences on $\hat G$, the discrete dual of $G$, we take the C$^*$-algebra
\[
  C^*_r(G) = c_0(\hat G) = c_0\mhyph \smashoperator[l]{\bigoplus_{i \in \Irr(G)}} \cB(H_i),
\]
where $c_0\mhyph\bigoplus$ denotes the C$^*$-algebraic direct sum.

As usual, without finite dimensionality assumption on $H$, an unitary representation of $G$ is given by an unitary element $U \in \cM(\cK(H) \otimes C(G))$ satisfying~\eqref{eq:fin-dim-unitar-rep}, where we extend $\id \otimes \Delta$ to a unital $*$-homomorphism $\cM(\cK(H) \otimes C(G)) \to \cM(\cK(H) \otimes C(G) \otimes C(G))$.
In particular, there is a distinguished representation (the \emph{regular representation}) of $G$ on $L^2(G)$, given by the \emph{multiplicative unitary} $W_G \in \cM(c_0(\hat G) \otimes C(G)) \subset \cM(\cK(L^2(G)) \otimes C(G))$.

Purely algebraic models of functions on these quantum groups are given as follows.
The \emph{regular subalgebra} of $C(G)$, or the algebra of \emph{matrix coefficients}, is denoted by $\cO(G)$, which is spanned by the elements $(\omega \otimes \id)(U)$ for $U \in \Rep(G)$ and $\omega \in \cB(H_U)_*$.
On the dual side, we take
\[
  c_{c}(\hat G) = \bigoplus_{i \in \Irr(G)} \cB(H_i)
\]
which is the usual algebraic direct sum.
There is a nondegenerate linear duality pairing between $\cO(G)$ and $c_c(\hat G)$, so that the algebra structures and the coalgebra structures on the spaces are related by
\begin{align}\label{eq:dual-hopf-alg-pairing}
  (\phi_{(1)}, f_1) (\phi_{(2)}, f_2) & = (\phi, f_1 f_2), & (\phi_1 \phi_2, f) & = (\phi_1, f_{(1)}) (\phi_2, f_{(2)}).
\end{align}

The \emph{Drinfeld double} of $G$ is represented by $\cO_c(\hat D(G))$, which is modeled on the tensor product coalgebra $c_c(\hat G) \otimes \cO(G)$ endowed with the coproduct $\Delta(\omega a) = \omega_{(1)} a_{(1)} \otimes \omega_{(2)} a_{(2)}$, and the $*$-algebra structure induced by those of $\cO(G)$ and $c_c(\hat G)$ together with the exchange rule
\[
  (a_{(1)} \rhd \omega) a_{(2)} = a_{(1)} (\omega \lhd a_{(2)}) \quad (\omega \in c_c(\hat G), a \in \cO(G)),
\]
with $a \rhd \omega$, $\omega \lhd a \in c_c(\hat{G})$ given by $(a \rhd \omega)(b) = \omega(b a)$ and $(\omega \lhd a)(b) = \omega(a b)$.
To be precise, this is to be interpreted as an algebraic model of convolution algebra of functions on the Drinfeld double quantum group $D(G)$.
Similarly, the function algebra on $D(G)$ is given by the tensor product $*$-algebra $\cO_c(D(G)) = \cO(G) \otimes c_c(\hat G)$ together with the twisted coproduct induced by the adjoint action of $W_G$, see~\cite{MR2566309}.
These algebras have completion as C$^*$-algebraic models for locally compact quantum groups, but we refrain from working with them.

\subsection{Quantum group actions}

A \emph{continuous action} of $G$ on a C$^*$-algebra $A$ is given by a nondegenerate and injective $*$-homomorphism $\alpha\colon A \to C(G) \otimes A$ satisfying $(\Delta \otimes \id) \alpha = (\id \otimes \alpha) \alpha$.
We say that $A$ is a \emph{$G$-C$^*$-algebra}.

Given a $G$-C$^*$-algebra $(A, \alpha)$, its \emph{regular subalgebra} $\cA$ is defined as the set of elements $a \in A$ such that $\alpha(a)$ belongs to the algebraic tensor product $\cO(G) \otimes A$.
This is a left $\cO(G)$-comodule algebra.

\begin{example}\label{ex:adj-action}
Let $(H, U)$ be a finite dimensional unitary representation of $G$.
Under our convention, $H$ is a right $\cO(G)$-comodule, but it can be considered as a left comodule by the coaction map $\xi \mapsto U_{21}^* (1 \otimes \xi)$.
This becomes an equivariant right Hilbert $\C$-module by the inner product
\[
  \langle \xi, \eta \rangle_\C = (U_1 \eta, U_{1'} \xi)_H h(U_{2'}^* U_2) = (\rho_U^{-1} \eta, \xi)_H
\]
where $(\eta, \xi)_H$ denotes the original inner product on $H$, see~\cite{MR3933035}.
Then $\cB(H)$ admits the induced $G$-C$^*$-algebra structure, concretely given by the coaction $T \mapsto U^*_{2 1} T_2 U_{2 1}$.
When $(H, U)$ is an infinite dimensional unitary representation, the same formula makes $\cB(H)$ a $G$-von Neumann algebra, and we also obtain a continuous a $G$-C$^*$-algebra $\cR(\cB(H))$, defined as
\[ 
\cR(\cB(H)) := \overline{ \{T \in \cB(H) \ | \ U^*_{2 1} T_2 U_{2 1} \in \cO(G) \odot \cB(H) \}} \ ,
\]
called the {\em regular subalgebra} of $\cB(H)$ (the closure is taken with respect to the norm topology).

\end{example}

A \emph{Yetter--Drinfeld $G$-C$^*$-algebra} is given by a $G$-C$^*$-algebra $(A, \alpha)$ together with a left action
\[
  \cO(G) \otimes \cA \to \cA, \quad f \otimes a \mapsto f \rhd a,
\]
satisfying
\begin{align}\label{eq:O(G)-moduleCstaralgebra}
  f \triangleright a^* = (S(f)^* \triangleright a)^*,
\end{align}
and such that
\[
  \alpha(f \rhd a) = f_{(1)} a_{(1)} S(f_{(3)}) \otimes (f_{(2)} \rhd a_{(2)}).
\]
Equivalently, it can be interpreted as a continuous action of the locally compact quantum group $D(G)$.
Because of this, we also say \emph{$D(G)$-equivariance} instead of \emph{Yetter--Drinfeld $G$-equivariance}.

Let $A$ be a left $\cO(G)$-module algebra.
By duality, we have a homomorphism
\[
  \beta_A\colon A \to \smashoperator[l]{\prod_{i \in \Irr(G)}} A \otimes \cB(H_i), \quad a \mapsto \biggl(\sum_{k,l} (u^{(i)}_{kl} \triangleright a) \otimes m^{(i)}_{kl}\biggr)_i,
\]
which can be regarded as a right comodule algebra over the multiplier Hopf algebra representing $\hat G$ with the convention of~\eqref{eq:dual-hopf-alg-pairing}.
If $A$ is a left $\cO(G)$-module C$^*$-algebra, meaning (\ref{eq:O(G)-moduleCstaralgebra}) holds, then $\beta_A$ is a $*$-homomorphism from $A$ to the multiplier C$^*$-algebra
\[
  \cM(A \otimes c_0(\hat{G})) \simeq \ell^{\infty} \mhyph \smashoperator[l]{\prod_{i \in \Irr(G)}} A \otimes \cB(H_i) .
\]

\medskip
By an analogue of the Tannaka--Krein--Woronowicz duality, we have an equivalence between the category of unital $G$-C$^*$-algebras and the category of pointed cyclic right C$^*$-$\Rep(G)$-module categories~\citelist{\cite{ostrik03}\cite{ydk1}\cite{n1}}.

Concretely, given a unital $G$-C$^*$-algebra $B$, one takes the category $\cD_B$ of finitely generated projective $G$-equivariant right Hilbert modules over $B$.
Thus, an object of $\cD_B$ is a right Hilbert module $E_B$ with a left coaction  $\delta$ of $C(G)$, such that the action of $B$ is equivariant.
Given an object $U$ of $\Rep(G)$, its right action $E_B \stmryolt U$ is represented by the equivariant right Hilbert module $H_U \otimes E_B$, where the underlying left comodule is given as the tensor product of $E_B$ and the left comodule $H_U$ as explained in Example~\ref{ex:adj-action}.
Explicitly, the new coaction on $E_B \stmryolt U$ is given by
\[
  \xi \otimes a \mapsto \left( U_{21}^*(1 \otimes \xi \otimes 1)\right)\delta(a)_{13}.
\]
The distinguished object in $\cD_B$ is $B$ as a right Hilbert C$^*$-module over itself.

Conversely, given a right C$^*$-$\Rep(G)$-module category $\cD$ and an object $X \in \cD$, we take the left $\cO(G)$-comodule
\[
  \cB_{\cD, X} = \bigoplus_{i \in \Irr(G)} \bar H_i \otimes \cD(X, X \stmryolt U_i),
\]
which admits an associative product from irreducible decomposition of monoidal products.
Together with the involution coming from duality of representations, we obtain a pre C$^*$-algebra which admits a canonical completion supporting a coaction of $C(G)$.

\begin{remark}\label{rem:leftcoactionsandrightmodules}
The formula above explains why left $C(G)$-comodule structures give rise to right $\Rep(G)$-module categories.
Indeed, given two finite dimensional unitary representations $(H_U, U)$ and $(H_V,V)$ of $G$, and given $\xi \in H_U$, $\eta \in H_V$ and $a \in E_B$, the $C(G)$-coaction on $\xi \otimes \eta \otimes a \in (H_V \otimes E_B) \stmryolt U$ is given by
\[
  \xi \otimes \eta \otimes a \mapsto U_{21}^* V_{31}^* (a_{(1)} \otimes \xi \otimes \eta \otimes a_{(2)}),
\]
where $a_{(1)} \otimes a_{(2)} = \delta(a)$.
Flipping the second and third legs, we obtain
\[
  U_{21}^* V_{31}^* ( a_{(1)} \otimes \xi \otimes \eta \otimes a_{(2)}) \mapsto (V \otimes U)_{21}^* (a_{(1)} \otimes (\eta \otimes \xi) \otimes a_{(2)}).
\]
This computation shows that the flip map $H_U \otimes H_V \to H_V \otimes H_U$ induces an equivariant isomorphism $(H_V \otimes E_B) \stmryolt U \simeq H_{V \otimes U} \otimes E_B$.
\end{remark}

The generalization of module categories to the nonunital setting is given by multiplier module categories~\cite{AV1}.

\section{Centrally pointed bimodule categories}\label{sec:centr-br-bimod-cat}

Fix a rigid C$^*$-tensor category $\cC$.
Then a \emph{$\cC$-bimodule category} $\cM$ is a category $\cM$ equipped with both a left and right $\cC$-module structures and with three sets of module associativity morphisms
\begin{gather*}
  \Psi'\colon (V \otimes W) \stmryogt X \to V \stmryogt (W \stmryogt X), \qquad \Psi\colon (X \stmryolt V) \stmryolt W \to X \stmryolt (V \otimes W) \\
  \Psi''\colon (V \stmryogt X) \stmryolt W \to V \stmryogt (X \stmryolt W)
\end{gather*}
for $X \in \cM$ and $V, W \in \cC$, satisfying the pentagon-type equations with the associativity morphism of $\cC$.
A \emph{$\cC$-bimodule functor} $F \colon \cM \to \cM'$ is given by a linear functor of underlying categories, together with
\begin{equation*}
  \tensor[_2]F{_{V,X}}\colon V \stmryogt F(X) \to F(V \stmryogt X), \qquad F_{2;X,V}\colon F(X) \stmryolt V \to F(X \stmryolt V),
\end{equation*}
again compatible with structure morphisms of $\cM$ and $\cM'$.

In the above, if the $\cM$ is both a left and a right C$^*$-$\cC$-module category, we say that it is C$^*$-$\cC$-bimodule if the module associative morphisms $\Psi,\Psi'$ and $\Psi''$ are unitaries.
If $\cM$ and $\cM'$ are C$^*$-$\cC$-bimodule categories, a $\cC$-bimodule functor $F\colon\cM \to \cM'$ a a morphism of C$^*$-$\cC$-bimodule categories if it is $*$-preserving and if $\tensor[_2]F{}$ and $F_2$ are unitaries.
We have thus defined a category of C$^*$-$\cC$-bimodule categories.

Alternatively, a C$^*$-$\cC$-bimodule category is nothing but a right C$^*$-$(\cC^{mp} \boxtimes \cC)$-module category, where $\cC^{mp}$ denotes the monoidal opposite of $\cC$ and $\boxtimes$ is the Kelly-Deligne tensor product.
Observe that $\cC^{mp} \boxtimes \cC$ is a rigid C$^*$-tensor category if $\cC$ is, and has simple unit if the unit of $\cC$ is simple.

To simplify our presentation we further assume that $\cC$ is strict.

\begin{definition}
Let $\cM$ be a C$^*$-$\cC$-bimodule category, and $m \in \cM$.
We say that $m$ is \emph{central} if there is a family of unitary morphisms
\[
  \sigma_U \colon U \stmryogt m \to m \stmryolt U
\]
that is natural in $U$ and satisfy the braid relations
\[
  \sigma_{U \otimes V} = (\sigma_U \stmryolt \id_V)(\id_U \stmryogt \sigma_V),
\]
for all $U,V \in \cC$.
We call $(\cM, m)$ a \emph{centrally pointed bimodule category}.
We call it a \emph{centrally cyclic bimodule category} when furthermore $m$ is a generator of $\cM$ under the action of $\cC^{mp} \boxtimes \cC$.
\end{definition}

\begin{remark}
In the above definition, because of centrality, $m$ is a generator of $\cM$ as a bimodule category if and only if it generates $\cM$ under either the right or left action of $\cC$ on $\cM$.
\end{remark}

\begin{definition}
Let $(\cM, m)$ and $(\cM', m')$ be centrally pointed bimodule categories.
A \emph{central functor} between them is a bimodule functor $F \colon \cM \to \cM'$ endowed with an isomorphism $F_0 \colon m' \to F(m)$ which are compatible with central generators and structure morphisms, in the sense that the diagram
\begin{equation}\label{eq:centr-br-ftr}
  \begin{tikzcd}
    U \stmryogt m' \arrow[r, "\sigma'_U"] \arrow[d] & m' \stmryolt U \arrow[d]\\
    U \stmryogt F(m) \arrow[d,"{\tensor[_2]{F}{}}"'] & F(m) \stmryolt U \arrow[d,"F_2"]\\
    F(U \stmryogt m) \arrow[r,"F(\sigma_U)"] & F(m \stmryolt U)
  \end{tikzcd}
\end{equation}
is commutative.
If $\cM$ and $\cM'$ are C$^*$-$\cC$-bimodule categories, $F$ is $*$-preserving and $F_0$ is unitary, we say that $F$ is a functor of central C$^*$-$\cC$-bimodule categories or that it is a bimodule central bimodule $*$-functor.
\end{definition}

\begin{definition}\label{def:nat-trans-bimod-functors}
Let $F$ and $F'$ be central functors from $(\cM, m)$ to $(\cM, m')$.
A \emph{natural transformation of bimodule functors} $\alpha\colon F \to F'$ is given by a natural transformation of the underlying linear functors $\alpha_X \colon F(X) \to F'(X)$ for $X \in \cM$ such that the diagrams
\[
  \begin{tikzcd}
    U \stmryogt F(X) \arrow[r, "{\tensor[_2]{F}{}}"] \arrow[d, "\id_U \otimes \alpha_X"'] & F(U \stmryogt X) \arrow[d, "\alpha_{U \stmryogt X}"]\\
    U \stmryogt F'(X) \arrow[r, "{\tensor[_2]{F}{^\prime}}"] & F'(U \stmryogt X)
  \end{tikzcd}
  \quad
  \begin{tikzcd}
    F(X) \stmryolt U \arrow[r, "F_2"] \arrow[d, "\alpha_X \otimes \id_U"'] & F(X \stmryolt U) \arrow[d, "\alpha_{X \stmryolt U}"]\\
    F'(X) \stmryolt U \arrow[r, "F'_2"] & F'(X \stmryolt U)
  \end{tikzcd}
  \quad
  \begin{tikzcd}
    m' \arrow[r, "F_0"] \arrow[dr, "F'_0"'] & F(m) \arrow[d, "\alpha_m"]\\
    & F'(m)
  \end{tikzcd}
\]
commute.
\end{definition}

\begin{definition}\label{def:centrallypointedcyclicbimodules}
Denote by $\CB(\cC)$ the category of centrally pointed C$^*$-$\cC$-bimodule categories and central bimodule $*$-functors, and by $\cCB(\cC)$ the subcategory of cyclic centrally pointed C$^*$-$\cC$-bimodule categories.
\end{definition}

Given a central object as above, we often work with the induced maps $\Sigma_{U;V,W}$ from $\cM(m \stmryolt V,m \stmryolt W)$ to $\cM(m \stmryolt U \otimes V, m \stmryolt U \otimes W)$, defined by
\begin{equation}\label{eq:br-map}
  \Sigma_{U;V,W}(T) = (\sigma_U \stmryolt \id_W) (\id_U \stmryogt T) (\sigma_U^* \stmryolt \id_V).
\end{equation}

\begin{proposition}\label{prop:C-embeds-into-half-braided-mod-cat}
Let $(\cM, m)$ a centrally pointed bimodule category.
There is a embedding of centrally pointed bimodule categories $F \colon (\cC,1_\cC) \to (\cM,m)$ which sends $U$ to $m \stmryolt U$.
\end{proposition}

\begin{proof}
For simplicity we assume that the right action of $\cC$ on $\cM$ is strict.
Then we get a right $\cC$-module functor $F(U) = m \stmryolt U$ with $F_2 = \id_{m \stmryolt U}$.
We extend it to a bimodule functor by setting
\[
  {\tensor[_2]{F}{}} = \sigma_U \otimes \id_V\colon (U \stmryogt m) \stmryolt V \to m \stmryolt U \stmryolt V = m \stmryolt (U \otimes V).
\]
Then consistency conditions of ${\tensor[_2]{F}{}}$ follow from the braid relations for $\sigma$.
\end{proof}

\begin{definition}\label{def:C-hat-U}
Let $\cM$ be a right C$^*$-$\cC$-module category, and $m \in \cM$.
We denote by $\hat{\cM}_m$ the idempotent completion of $\cC$ with the enlarged morphism sets
\begin{equation}\label{eq:hat-M-morphism-set}
  \hat{\cM}_m(V,W) = \Nat_b ( m \stmryolt \iota \otimes V, m \stmryolt \iota \otimes W) \simeq \ell^\infty\mhyph\smashoperator[l]{\prod_{X \in \Irr(\cC)}} \cM(m \stmryolt X \otimes V, m \stmryolt X \otimes W).
\end{equation}
\end{definition}

This is a bimodule category, with $1$ (which corresponds to $m \in \cM$) being a central generator.
The bimodule structure of $\hat\cM_m$ is as follows.
At the level of the objects it is induced by the tensor structure of $\cC$ At the level of morphisms, for $T \in \hat\cM_m(V,W)$, $X,Y,Z \in \cC$ and $\phi \in \cC(X,Y)$, we define the $Z$-component of the right action $T \stmryolt  \phi$ of $\phi$ on $T$ by
\[
  (T \stmryolt  \phi)_Z = T_Z \stmryolt  \phi .
\]
On the right hand side of the above equation we have used the module structure of $\cM$.
The left action $\phi \stmryogt T$, on the other hand, has the $Z$ component defined by the commutative diagram
\begin{equation*}
  \begin{tikzcd}[row sep=small, column sep=small] m \stmryolt  Z \triangleleft(X \otimes V) \arrow[rrrr, "(\phi \stmryogt T)_Z"] \arrow[dd] &  &  &  & m \stmryolt  Z \stmryolt  (Y \otimes W) \\
  &  &  &  & \\
  m \stmryolt  (Z \otimes X) \stmryolt  V \arrow[rr, "T_{Z \otimes X}"'] &  & m \stmryolt  (Z \otimes X) \stmryolt  W \arrow[rr] &  & m \stmryolt  Z \stmryolt  (X \otimes W). \arrow[uu, "\id_{m \stmryolt  Z} \stmryolt  (\phi \otimes \id_W)"']
  \end{tikzcd}
\end{equation*}
Using the naturality of $T$, it is easy to check that these indeed define a $\cC$-bimodule structure on $\hat{\cM}_m$.
It is moreover compatible with the $*$-structure, so that $\hat{\cM}_m$ has a canonical structure of a C$^*$-$\cC$-bimodule category.
It is immediate that $1 \in \hat\cM_m$ is a central object, with $\sigma_U$ given by $\id_U$.

This is motivated by the `dual category' $\hat\cC$ introduced in~\cite{NY1}, which corresponds to case of $\cM = \cC$ and $m = 1_\cC$.
In this case there is a natural C$^*$-tensor structure on $\hat\cC$ (with nonsimple unit) such that $U \mapsto U$ is a C$^*$-tensor functor from $\cC$ to $\hat\cC$.
If we take $m = U$ instead, because of the centrality, the resulting category can be identified with $\hat\cC_U$, with morphism sets
\[
  \hat\cC_U(V, W) = \hat\cC(U\otimes V,U \otimes W) = \ell^\infty\mhyph\smashoperator[l]{\prod_{i \in \Irr(\cC)}} \cC(U_i \otimes U \otimes V, U_i \otimes U \otimes W)
\]
for $V, W \in \cC$.
In terms of compact quantum group actions, we have the following.

\begin{proposition}
Suppose $\cC = \Rep(G)$ for some compact quantum group $G$.
For a finite dimensional unitary representation $(H, U)$ of $G$, the underlying right C$^*$-$\Rep(G)$-module category of $\hat{\cC}_U$ corresponds to the $G$-C$^*$-algebra $\cR(\cB(H) \vnotimes \ell^\infty(\hat{G}))$, the regular part of the $G$-W$^*$-algebra $\cB(H) \vnotimes \ell^{\infty}(\hat{G})$.
\end{proposition}

\subsection{Trivialization of central structure}

Let us consider a construction that we call \emph{trivialization of the central structure}.
Let $(\cM,m, \sigma)$ be a centrally pointed $\cC$-bimodule $C^*$-category.
Then, as usual, we consider the idempotent completion of $\cC$ with morphism sets $\Mor(V, W) = \cM(m \stmryolt V, m \stmryolt W)$.
In other words, our objects are projections $p$ in the $C^*$-algebras $\cM(m \stmryolt V)$.
Given such $p$, and another $U \in \cC$, consider the projection
\[
  \Sigma_{U;V}(p) = (\sigma_U \stmryolt \id_V) (\id_U \stmryogt p) (\sigma_U^{-1} \stmryolt \id_V) \in \cM(m \stmryolt U \otimes V),
\]
which defines a subobject of $U \otimes V$ in our category.
Of course, this is isomorphic to the left action of $U$ on the subobject of $m \stmryolt V$ defined by $p$, i.e., isomorphic to the subobject of $U \stmryogt m \stmryolt V$ defined by the projection $\id_U \stmryogt p$.
This way (assuming $\cC$ is strict) we get a strict bimodule category.
Moreover, $1$ is a central object, with the structure morphisms $\sigma'_U = \id_U$.

\begin{proposition}\label{prop:equiv-to-triv-braiding-cat}
The centrally pointed bimodule category $\tilde\cM$ obtained by trivialization of the central structure is isomorphic to the original one.
\end{proposition}

\begin{proof}
Define $F_\sigma\colon \tilde\cM \to \cM$ as (the canonical extension of) the functor $V \mapsto m \stmryolt V$ at the level of objects, and by the tautological map at the level of morphisms.
This is a right module functor.
The compatibility with left module structures,
\[
  \tensor[_2]{(F_\sigma)}{}\colon V \stmryogt m \stmryolt W = V \stmryogt F_\sigma(W) \to m \stmryolt V \stmryolt W = F_\sigma(V \otimes W),
\]
is given by $\sigma_V \otimes \id_W$.
Then the commutativity of~\eqref{eq:centr-br-ftr} holds by construction.

In the other direction, to define $F \colon \cM \to \tilde\cM$, we choose $V_X \in \cC$ and an isometry $j_X \colon X \to m \stmryolt V_X$ for each $X \in \cM$.
At the level of C$^*$-functor, for each $X \in \cM$ the object $F(X)$ is represented by the projection $j_X j_X^* \in \tilde\cM(V_X)$, while for each $T \in \cM(X, Y)$ the morphism $F'(T)$ is represented by $j_Y T j_X^* \in \tilde\cM(V_X, V_Y)$.

Then, we can define the natural isomorphisms giving compatibility with bimodule structure as follows.
For the right module structure, we take
\[
  F_2 = j_{X \stmryolt U} (j_X^* \stmryolt \id_U) \colon F'(X) \stmryolt U \to F(X \stmryolt U)
\]
while for the left module structure we take
\[
  \tensor[_2]{F}{^\prime} = j_{U \stmryogt X} (\id_U \stmryogt j_X^*) (\sigma_U^* \stmryolt \id_{V_X}) \colon U \stmryogt F(X) \to F(U \stmryogt X).
\]

Then the composition $F_\sigma F \colon \cM \to \cM$ is naturally isomorphic to the identity functor $\Id_\cM$.
Indeed, the assumption on cyclicity says that $\cM \ni X \mapsto (m \stmryolt V_X, j_x j_X^*)$ extends to an equivalence of right C$^*$-$\cC$-module categories.
Under this equivalence, $F_\sigma F$ corresponds to the identity functor on $\cM$, and this correspondence is by construction compatible with the module structures.
A similar argument shows that $F F_\sigma$ is naturally isomorphic to the identity functor on $\tilde{\cM}$.
\end{proof}

\begin{proposition}\label{prop:strictify-mod-ftr}
Let $\cM_1$ and $\cM_2$ be $\cC$-bimodule categories with central generators $m_1$ and $m_2$, respectively.
Let $F \colon \cM_1 \to \cM_2$ be a functor of right $\cC$-module categories such that $F(m_1) = m_2$.
Suppose that the maps~\eqref{eq:br-map} satisfy
\begin{multline}\label{eq:right-mod-ftr-compat-with-braiding}
  F_{2;m_1,U\otimes W}^{-1} F(\Sigma_{U;V,W}(T)) F_{2;m_1,U \otimes V} \\
  = \Sigma_{U;V,W}'(F_{2;m_1,W}^{-1}F(T)F_{2;m_1,V}) \quad (T \in \cM_1(m_1 \stmryolt V, m_1 \stmryolt W).
\end{multline}
Then we have a strict bimodule functor $\tilde F\colon \tilde\cM_1 \to \tilde\cM_2$ characterized by
\[
  \tilde F(T) = F_{2;m_1,W}^{-1} F(T) F_{2;m_1,V}
\]
for morphisms $T \in \tilde\cM_1(V, W) = \cM_1(m_1 \stmryolt V, m_1 \stmryolt W))$.
\end{proposition}

\begin{proof}
By the above formula, an object of $\tilde\cM_1$ represented by a projection $p \in \cM_1(m_1 \stmryolt V)$ is mapped to the object of $\tilde\cM_2$ represented by the projection $F_{2;m_1,V}^{-1} F(p) F_{2;m_1,V}$.

Then the above claim says that $\tilde F$ is a strict functor of left $\Rep(G)$-module categories.
The compatibility of $F_2$ and the (trivial) associativity morphisms imply that $\tilde F$ is also a strict functor of right module categories.
\end{proof}

The above functor $\tilde F$ satisfies $\tilde F(\tilde\sigma_U) = \tilde\sigma'_U$ for the half-braidings of $\tilde\cM_1$ and $\tilde\cM_2$, since $\tilde\sigma_U$ is given by $\id_{m_1 \stmryolt U} \in \cM(m_1 \stmryolt U)$.
Combined with Proposition~\ref{prop:equiv-to-triv-braiding-cat}, we obtain a central functor $\cM_1 \to \cM_2$ from a right module functor satisfying~\eqref{eq:right-mod-ftr-compat-with-braiding}.
Consequently, $\CB(\cC)$ is equivalent to the category with the same objects but right module category functors satisfying~\eqref{eq:right-mod-ftr-compat-with-braiding} as morphisms.

\section{Duality for Yetter--Drinfeld \texorpdfstring{$C^*$}{C*}-algebras}\label{sec:duality-for-YD-algs}

In this section, we establish the Tannaka--Krein duality for the unital Yetter--Drinfeld $G$-C$^*$-algebras.
In our presentation, their categorical dual are bimodules over $\Rep(G)$ with central generators.

\begin{definition}
    Given a compact quantum group $G$, we denote the category of unital Yetter--Drinfeld $G$-C$^*$-algebras and equivariant $*$-homomorphisms by $\YD(G)$.
\end{definition}

Recall also Definition~\ref{def:centrallypointedcyclicbimodules}, where the category $\CB^c(\Rep(G))$ of centrally pointed cyclic C$^*$-$\Rep(G)$-bimodules is defined.

\begin{theorem}\label{thm:yetterdrinfeldtheo1}
The categories $\YD(G)$ and $\cCB(\Rep(G))$ are equivalent.
Under this equivalence, two morphisms $F, F' \colon \cM \to \cM'$ in $\cCB(\Rep(G))$ induce the same homomorphism $B_\cM \to B_{\cM'}$ if and only if they are naturally isomorphic as central bimodule functors.
\end{theorem}

The proof of this theorem will occupy the rest of the section.

\subsection{From Yetter--Drinfeld algebras to centrally pointed bimodules}

Let $B$ be a Yetter--Drinfeld $C^*$-algebra with $G$-action $\alpha$ and $\hat{G}$-action $\beta$.
Let us show that the category $\cD_B$ admits a structure of centrally pointed bimodule category.

Given a unitary finite dimensional representation $U$ of $G$, we fix a choice $\{m^{U}_{ij}\}_{i,j}$ of matrix units for $\cB(H_U)$ and write
\[
  U = \sum_{i,j} m^U_{ij} \otimes u_{ij}.
\]
Also, on the right Hilbert $B$-module $H_U \otimes B$ we consider the coaction of $C(G)$ given by
\[
  \delta_U (\xi \otimes b) = U_{21}^* (b_{(1)} \otimes \xi \otimes b_{(2)}),
\]
where $\alpha(b) = b_{(1)} \otimes b_{(2)}$ denotes the coaction of $G$ on $B$.

\begin{proposition}\label{prop:equivariancecanonicalleftaction}
For a finite dimensional unitary representation $U$ of $G$, consider the homomorphism $\pi_U\colon B \to \End_B(H_U \otimes B)$ given by
\[
  \pi_U(a) = \sum_{i,j} m^U_{ij} \otimes (u_{ij} \rhd a).
\]
Then we have the equivariance condition
\[
  \delta_U ( \pi_U(a) (\xi \otimes b) ) = \left( (\id \otimes \pi_U) \alpha(a) \right) (\delta_U (\xi \otimes b)) \quad (a, b \in B, \xi \in H_U).
\]
\end{proposition}

\begin{proof}
Let $\{ \xi_i \}_i$ be the basis of $H_U$ corresponding to the matrix units $\{m^U_{ij} \}_{ij}$.
Let us check the claim for $\xi = \xi_i$.

The left hand side is
\begin{align*}
  \delta_U (\pi_U(a) (\xi_i \otimes b)) & = \delta_U \left( \sum_{jk}m^U_{jk}(\xi_i) \otimes (u_{jk} \rhd a) b \right) \\
  & = \sum_{j,r,s} u_{rs}^* (u_{ji} \rhd a)_{(1)} b_{(1)} \otimes m^U_{rs}(\xi_j) \otimes (u_{ji} \rhd a)_{(2)} b_{(1)} \\
  & = \sum_{j,s,k,l} u_{js}^* u_{jk} a_{(1)} u_{il}^* b_{(1)} \otimes \xi_s \otimes (u_{kl} \rhd a_{(2)}) b_{(2)},
\end{align*}
since
\[
  \alpha (u_{ij} \rhd a) = (u_{ij} \rhd a)_{(1)} \otimes (u_{ij} \rhd a)_{(2)} = \sum_{k,l} u_{jr} a_{(1)} u_{il}^* \otimes (u_{kl} \rhd a_{(2)} )
\]
holds by the Yetter--Drinfeld property.
As $\sum_j u_{js}^* u_{jk} = \delta_{sk}$
\begin{align}\label{yetterdrinfeldeq1}
  \delta_U (\pi_U(a) (\xi_i \otimes b)) = \sum_{kl} a_{(1)} u_{il}^* b_{(1)} \otimes \xi_k \otimes (u_{kl} \rhd a_{(2)}) b_{(2)}
\end{align}
For the right hand side, we have
\begin{align*}
  \left( (\id \otimes \pi_U) \alpha(a) \right) \delta_U(\xi_i \otimes b) & = (a_{(1)} \otimes \pi_U(a_{(2)}) )\left( \sum_{j,k} u_{jk}^* b_{(1)} \otimes m^U_{kj}(\xi_i) \otimes b_{(2)} \right) \\
  & = \sum_k a_{(1)} u_{ik}^* b_{(1)} \otimes \pi_U(a_{(2)}) (\xi_k \otimes b_{(2)} ) \\
  & = \sum_{k,r,s} a_{(1)} u_{ik}^* b_{(1)} \otimes m^U_{rs}(\xi_k) \otimes (u_{rs} \rhd a_{(2)} )b_{(2)} \\
  & = \sum_{k,l} a_{(1)}u_{ik}^* b_{(1)} \otimes \xi_l \otimes (u_{lk} \rhd a_{(2)}) b_{(2)},
\end{align*}
which proves the claim.
\end{proof}

Thus, $H_U \otimes B$ becomes an equivariant bimodule over $B$.
We use this structure to define balanced tensor products
\[
  \left( H_U \otimes B \right) \otimes_B \left(H_V \otimes B \right),
\]
which is again a equivariant Hilbert $B$-module.
Note that the canonical isomorphism $B \otimes_B (H_U \otimes B ) \simeq H_U \otimes B$ is $G$-equivariant.
Given $U,V \in \Rep(G)$, the map
\[
  S_{U,V}\colon (H_V \otimes B) \otimes_B (H_U \otimes B) \to H_{U \otimes V} \otimes B, \quad (\zeta \otimes b) \otimes_B (\xi \otimes a) \mapsto \sum_{i,j} m^U_{ij}(\xi) \otimes \zeta \otimes (u_{ij} \rhd b)a
\]
is an equivariant isomorphism of right Hilbert $B$-modules.

\begin{proposition}\label{yetterdrinfeldprop2}
The map $S_{U,V}$ is a $B$-bimodule isomorphism.
\end{proposition}

\begin{proof}
We need to show that the map $S_{U,V}$ respect the left $B$-actions, which is equivalent to the equality
\[
  \pi_V (b) \otimes \id = S_{U,V}^{-1} \pi_{U \otimes V}(b) S_{UV}
\]
for $b \in B$.

Let $a,b,c \in B$ and $\xi_{i_0} \in H_U, \zeta_{k_0} \in H_V$ be basis elements corresponding to the respective matrix units of $U$ and $V$.
\begin{align*}
  \pi_{U \otimes V}(c) S_{U,V} \left( (\zeta_{k_0} \otimes b ) \otimes_B ( \xi_{i_0} \otimes a) \right) & = \pi_{U \otimes V}(c) \left( \sum_i \xi_i \otimes \zeta_{k_0} \otimes ( u_{ii_0} \rhd a) \right) \\
  & = \sum_{ij} \xi_i \otimes \zeta_k \otimes (u_{ij}v_{kk_0} \rhd c )(u_{ji_0} \rhd b) a.
\end{align*}
On the other hand, we have
\begin{align*}
  S_{U,V} \left( \pi_V(c) ( \zeta_{K_0} \otimes b) \otimes_B (\xi_{i_0} \otimes a) \right) & = S_{U,V} \left( \sum_{k} \left( \zeta_k \otimes (v_{kk_0} \rhd c) b \right) \otimes_B (\xi_{i_0} \otimes a) \right) \\
  & = \sum_{i,k} \xi_i \otimes \zeta_k \otimes \left( u_{ii_0} \rhd ((v_{kk_0} \rhd c)b)  \right) a \\
  & = \sum_{i,j,k} \xi_i \otimes \zeta_k \otimes (u_{ij} v_{kk_0} \rhd c)(u_{ji_0} \rhd b) a \\
  & = \pi_{U \otimes V} S_{U,V} \left( (\zeta_{k_0} \otimes b) \otimes_B (\xi_{i_0} \otimes a ) \right),
\end{align*}
which proves the claim.
\end{proof}

Write briefly $B_U = H_U \otimes B$.

\begin{proposition}\label{yetterdrinfeldprop3}
We have
\[
  S_{U \otimes V,W} (\id_{B_W} \otimes S_{U,V}) = S_{U,V \otimes W} S_{V,W} \otimes \id_{B_U}.
\]
\end{proposition}

\begin{proof}
Let us compare the action of both sides on vectors of the form $(\zeta_{k_0} \otimes c) \otimes (\eta_{j_0} \otimes b ) \otimes (\xi_{i_0} \otimes a)$.
The left hand side gives
\begin{multline*}
  S_{U \otimes V,W} \left( \sum_{i_1} (\zeta_{k_0} \otimes c) \otimes (\xi_{i_1} \otimes \eta_{j_0} \otimes (u_{i_1 i_0} \rhd b) a) \right) \\
  = \sum_{i_1,i_2,j} \xi_{i_1} \otimes \eta_j \otimes \zeta_{k_0} \otimes (u_{i_1i_2}v_{jj_0} \rhd c)(u_{i_2i_0} \rhd b) a = \sum_{ij} \xi_i \otimes \eta_j \otimes \zeta_{k_0} \otimes \left( u_{ii_0} \rhd (v_{jj_0} \rhd c) b \right) a.
\end{multline*}
The right hand side gives
\begin{equation*}
  S_{U,V \otimes W} \biggl( \biggl( \sum_j (\eta_j \otimes \zeta_{k_0}) \otimes (v_{jj_0} \rhd c) b \biggr)b \otimes (\xi_{i_0} \otimes a ) \biggr) = \sum_{ij} \xi_i \otimes \eta_j \otimes \zeta_{k_0} \otimes \left( u_{ii_0} \rhd (v_{jj_0} \rhd c) b \right) a,
\end{equation*}
hence we obtain the claim.
\end{proof}

Given $W \in \Rep(G)$, let us denote by $\beta_W$ the composition of the $\hat{G}$-action $\beta\colon B \to \cM (B \otimes c_0(\hat{G}) )$ with the projection $\id \otimes \pi_W$, where $\pi_W \colon \ell^\infty(\hat{G}) \to \cB(H_W)$ is the representation homomorphism.

\begin{definition}\label{yetterdrinfelddef1}
Define a left action of $\Rep(G)$ on the category $\cD_B$ by
\[
  V \stmryogt X = X \otimes_B (H_V \otimes B),
\]
and for $T \in \cD_B (X, Y)$ and $W \in \Rep(G)$,
\[
  \id_W \stmryogt T = T \otimes \id_{H_W \otimes B}.
\]
\end{definition}

\begin{proposition}\label{yetterdrinfeldprop4}
Together with the left $\Rep(G)$-action given by Definition~\ref{yetterdrinfelddef1}, $\cC_B$ is a $\Rep(G)$-bimodule $C^*$-category.
\end{proposition}

\begin{proof}
Proposition~\ref{yetterdrinfeldprop3} gives natural isomorphisms
\[
  (U \stmryogt B) \stmryolt V = H_V \otimes (B \otimes_B B_U) \simeq B_V \otimes_B B_U  = U \stmryogt (B \stmryolt V).
\]
We need to show that for $T \in \Mor_{\Rep(G)}(U,V)$, $T \otimes \id_B \in \cC_B(B_U,B_V)$ is a left $B$-module map.
For $a,b \in B$ and $\xi \in H_U$,
\begin{align*}
  (T \otimes \id_B) \left( \pi_U(b) ( \xi \otimes a) \right) & = (T \otimes \id_B) \left( m^U_{ij}(\xi) \otimes ( u_{ij} \rhd b) a \right) \\
  & = (T \otimes \id_B) \left( m^U_{ij}(\xi) \otimes ((\id \otimes u_{ij}) \beta(b) \right) a.
\end{align*}
By the intertwiner condition for $T$, we see that this is equal to
\[
  m^V_{kl}(T\xi) \otimes \left( (\id \otimes v_{kl}) \beta(b) \right) a = \pi_V(b) (T\xi \otimes a),
\]
hence we obtain the claim.
\end{proof}

We conclude that the $\Rep(G)$-bimodule category $\cD_B$ has $B$ as a central generator, with the half-braiding $\sigma_U\colon U \stmryogt B \to B \stmryolt U$ given by the identity map up to the standard identification
\[
  U \stmryogt B = B \otimes_B (H_U \otimes B) \simeq H_U \otimes B.
\]
Thus, a Yetter--Drinfeld $G$-$C^*$-algebra gives rise to a $\Rep(G)$-bimodule category with a central generator.

\subsection{From centrally pointed bimodules to Yetter--Drinfeld algebras}

In the other direction, let $\cM$ be a $\Rep(G)$-bimodule $C^*$-category with a central generator $m$, with central structure $\sigma_U \colon U \stmryogt m \to m \stmryolt U$ for $U \in \Rep (G)$.
Being a cyclic right $\Rep(G)$-module $C^*$-category, to $\cM$ there is associated a unital $G$-$C^*$-algebra $B_\cM$, which is a completion of the vector space
\[
  \cB_\cM = \bigoplus_{U \in \Irr(G)} \bar{H}_U \otimes \cM(m,m \stmryolt U),
\]
which we call \emph{algebraic regular} part of $B_\cM$.
Let us briefly recall the $*$-algebra structure on this space.
For further reference, see~\cites{n1,NY3}.
It is built using the auxiliary universal algebra
\[
  \tilde{\cB}_\cM = \bigoplus_{U \in \cC}  \bar{H}_U \otimes \cM(m,m \stmryolt U),
\]
called  the \emph{universal cover} of $\cB_\cM$.
Its product is defined by
\[
  (\bar{\xi} \otimes T) \bullet (\bar{\eta} \otimes S) = (\overline{\xi \otimes \eta}) \otimes (T \otimes \id_V)S
\]
for $\bar{\xi} \otimes T \in \bar{H}_U \otimes \cM(m,m \stmryolt U)$ and $\bar{\eta} \otimes S \in \bar{H}_V \otimes \cM(m, m \stmryolt V)$.

There is a surjective idempotent linear map $\pi\colon \tilde{\cB}_\cM \to \cB_\cM$ defined by the decompositions of arbitrary elements of $\Rep(G)$ into irreducible objects, and for $x,y \in \cB_\cM$ the product $xy$ is $\pi(x \bullet y)$.
The specific form of $\pi$ will be given in the proof of Lemma~\ref{ydstructurelemma1}.
The involution on $\tilde{\cB}_\cM$ is given by
\[
  (\bar{\xi} \otimes T)^\dagger = \overline{\overline{ \rho_U^{-1/2} \xi}} \otimes (T^* \otimes \id_U)\bar{R}_U
\]
for $\bar{\xi} \otimes T \in \bar{H}_U \otimes \cM(m, m \stmryolt U)$, and we obtain the involution on $\cB_\cM$ by $x^* = \pi(x^\dagger)$.

The construction above can be applied to the Yetter--Drinfeld algebra $C(G)$.
The algebraic regular part of it is $\cO(G)$, the algebra of polynomial functions on $G$, and its universal cover is $\tilde{\cO}(G)$.
In this case, we denote by $\pi_G$ the canonical projection $\tilde{\cO}(G) \to \cO(G)$.

We claim that the $*$-algebra $\cB_\cM$ has a Yetter--Drinfeld structure.
Consider the action $\tilderhd \colon \tilde{\cO}(G) \otimes \tilde{\cB}_\cM \to \tilde{\cB}_\cM$ defined by
\begin{multline}\label{eq:reg-alg-action-in-YD-str}
  (\bar \xi \otimes \zeta) \tilde\rhd (\bar\eta \otimes T) \\
  = (\overline{\xi \otimes \eta \otimes \overline{\rho^{-1/2} \zeta}}) \otimes \left( (\sigma_U \otimes \id_V \otimes \id_{\bar{U}}) (\id_U \otimes T \otimes \id_{\bar{U}} ) (\sigma_U^{-1} \otimes \id_{\bar{U}} ) \bar{R}_U \right)
\end{multline}
for $T \in \cM(m, m \stmryolt V)$, $\bar\eta \in \bar H_V$, and $\xi, \zeta \in H_U$.
For $x  \in \cO(G)$ and $a \in \cB_\cM$, we put $x \rhd a = \pi ( x \mathbin{\tilde\rhd} a)$.

From now on, $U,V$ and $W$ will always stand for finite dimensional unitary representations of $G$, and to avoid cumbersome formulas we will often write $U V$ instead of $U \otimes V$.

\begin{lemma}\label{ydstructurelemma1}
For all $x \in \tilde{\cO}(G)$ and every $a \in \tilde{\cB}_\cM$,
\[
  \pi_G(x) \rhd \pi(a) = \pi( x \tilderhd a) .
\]
\end{lemma}

\begin{proof}
Take $x = \bar{\xi} \otimes \zeta \in \bar{H}_U \otimes H_U$ and $a = \bar{\eta} \otimes T \in \bar{H}_V \otimes \cM (m, m \stmryolt V)$.
Take partial isometries $u_i\colon H_{s_i} \to H_U$ and $v_j\colon H_{s_j} \to H_V$ defining decompositions of $U$ and $V$ into irreducible finite dimensional unitary representations.
Then we have
\[
  \pi_G(x) \rhd \pi(a) =  \pi \left( \sum_{i,j} (\overline{u_i^* \xi} \otimes u_i^* \zeta) \tilderhd (\overline{v_j^* \eta} \otimes v_j^* T ) \right),
\]
which can be expanded as
\[
  \pi \left( \sum_{i,j} \overline{( u_i^* \xi \otimes v_j^* \eta \otimes \overline{\rho_U^{-1/2} (u_i^* \zeta)})} \otimes (\sigma_{U_{s_i}} \otimes \id_V \otimes \id_{\bar{U}_{s_i}}) ( \id_{U_{s_i}} \otimes v_j^* T \otimes \id_{\bar{U}_{s_i}})( \sigma_{U_{s_i}}^* \otimes \id_{\bar{U}_{s_i}}) \bar{R}_{s_i} \right).
\]
On the other hand, we also have
\[
  \pi ( x \tilde{ \rhd} a) = \pi \left( \overline{( \xi \otimes \eta \otimes \overline{ \rho_U^{-1/2} \zeta} ) }  \otimes (\sigma_U \otimes \id_V \otimes \id_{\bar{U}}) ( \id_U \otimes T \otimes \id_{\bar{U}}) ( \sigma_U^* \otimes \id_{\bar{U}}) \bar{R}_U   \right)
\]
which can be expanded as
\begin{multline*}
  \pi \Bigl( \sum_{i,j,k} \overline{ ( u_i^* \xi \otimes v_j^* \eta \otimes \bar{u}_k^* \overline{ \rho_U^{-1/2} \zeta})} \otimes \\
  ( \sigma_{U_{s_i}} \otimes \id_{U_{s_j}} \otimes \id_{\bar{U}_{s_k}}) (\id_{U_{s_i}} \otimes v_j^* T \otimes \id_{\bar{U}_{s_k}}) (\sigma_{U_{s_i}} \otimes \id_{\bar{U}_{s_k}})( u_i^* \otimes u_k^*) \bar{R}_U  \Bigr).
\end{multline*}
Since $u_k^* \rho_U = \rho_{U_k} u_k^*$, $\bar{R}_U = \sum_i (u_i \otimes \bar{u}_i) \bar{R}_{s_i}$ and the partial isometries $u_i$ have mutually orthogonal ranges, we conclude that
\[
  \pi_G(x) \rhd \pi(a) = \pi ( x \tilde{ \rhd} a),
\]
establishing the claim.
\end{proof}

\begin{lemma}\label{ydstructurelemma2}
The map $\rhd$ makes $\cB_\cM$ into a left $\cO(G)$-module.
\end{lemma}
\begin{proof}
Taking $x = \bar{\xi} \otimes \zeta \in \bar{H}_U \otimes H_U$, $y = \bar{\mu} \otimes \nu \in \bar{H}_W \otimes H_W$ and $a = \bar{\eta} \otimes T \in \bar{H}_V \otimes \cM (m, m \stmryolt V)$, we must show that
\[
  \pi ( x \tilderhd (y \tilderhd a)) = \pi ((xy) \tilderhd a) .
\]

Expanding the left hand side, we obtain
\begin{multline*}
  \pi \Bigl(  \overline{ ( \xi \otimes \mu \otimes \eta \otimes \overline{\rho_W^{-1/2} \nu} \otimes \overline{ \rho_U^{-1/2} \zeta})} \\
  \otimes (\sigma_{UW} \otimes \id_{V\bar{W} \bar{U}}) ( \id_{UW} \otimes T \otimes \id_{\bar{W} \bar{U}})( \sigma_{UW}^* \otimes \id_{\bar{W} \bar{U}}) ( \id_U \otimes \bar{R}_W \otimes \id_{\bar{U}}) \bar{R}_U   \Bigr) .
\end{multline*}
The right hand side is, in turn,
\begin{equation*}
  \pi \left(\overline{( \xi \otimes \mu \otimes \overline{ \rho_U^{-1/2} \zeta \otimes \rho_W^{-1/2} \nu})} \otimes( \sigma_{UW} \otimes \id_{V \bar{UW}})( \id_{UW} \otimes T \otimes \id_{\bar{UW}})(\sigma_{UW}^* \otimes \id_{\bar{UW}}) \bar{R}_{UW} \right)
\end{equation*}
The equality between the two expressions comes from the fact that
\[
  \gamma\colon \bar{H}_W \otimes \bar{H}_U \to \overline{H_U \otimes H_W}, \quad \gamma(\bar{a} \otimes \bar{b}) = \overline{b \otimes a},
\]
is an equivalence of representations, and that we can take
\[
  \bar{R}_{UW} = (\id_{UW} \otimes \gamma) (\id_U \otimes \bar{R}_W \otimes \id_{\bar{U}}) \bar{R}_U,
\]
hence the claim.
\end{proof}

\begin{lemma}\label{ydstructurelemma3}
The map $\rhd$ is an algebra-module map, that is,
\[
  x \rhd(ab) = (x_{(1)} \rhd a) (x_{(2)} \rhd b)
\]
holds for any $x \in \cO(G)$ and $a, b \in \cB_\cM$.
\end{lemma}
\begin{proof}
Take $a = \bar{\eta} \otimes T \in \bar{H}_V \otimes \cM(m, m \stmryolt V)$, $b = \bar{\zeta} \otimes S \in \bar{H}_W \otimes \cM(m, m \stmryolt W)$, and take $x$ to be $u_{ij} = \pi_G(\bar{\xi}_i \otimes \xi_j)$, where $\{\xi_i\}_i$ is an orthonormal basis of $H_U$, so that $\{u_{ij}\}$ are the matrix coefficients of $U$ with respect to it.
By $\Delta(u_{ij}) = \sum_k u_{ik} \otimes u_{kj}$, our claim is equivalent to
\[
  \pi \left( (\bar{\xi}_i \otimes \xi_j) \tilderhd (ab) \right) = \sum_k \pi \left( ((\bar{\xi}_i \otimes \xi_k) \tilderhd a)((\bar{\xi}_k \otimes \xi_j) \tilderhd b) \right).
\]
The left hand side is
\[
  \pi \left(\overline{(\xi_i \otimes \eta \otimes \zeta \otimes \overline{\rho_U^{-1/2} \xi_j})} \otimes (\sigma_U \otimes \id_{VW\bar{U}})(\id_U \otimes (T \otimes \id_W)S \otimes \id_{\bar{U}})(\sigma_U^* \otimes \id_{\bar{U}}) \bar{R}_U \right),
\]
while the right hand side is
\[
  \sum_k \pi \left( \overline {(\xi_i \otimes \eta \otimes \overline{\rho_U^{-1/2} \xi_k} \otimes \xi_k \otimes \zeta \otimes \overline{\rho_U^{-1/2} \xi_j})} \otimes X \right),
\]
where
\begin{multline*}
  X = (\sigma_U \otimes \id_{V\bar{U}UW\bar{U}})(\id_U \otimes T \otimes \id_{\bar{U}UW \bar{U}})\\
  (\sigma_U^* \otimes \id_{\bar{U}UW\bar{U}}) (\bar{R}_U \otimes \id_{UW\bar{U}})(\sigma_U \otimes \id_{W\bar{U}})(\id_U \otimes S \otimes \id_{\bar{U}})(\sigma_U^* \otimes \id_{\bar{U}}) \bar{R}_U).
\end{multline*}
We have that $\sum_k \overline{\rho_U^{-1/2} \xi_k} \otimes \xi_k = R_U(1)$, and the operator $d_U^{-1/2} R_U$ is an isometric embedding of the trivial representation $1$ into $\bar{U} \otimes U$.
Therefore
\[
  \sum_k \pi \left( \overline {(\xi_i \otimes \eta \otimes \overline{\rho_U^{-1/2} \xi_k} \otimes \xi_k \otimes \zeta \otimes \overline{\rho_U^{-1/2} \xi_j})} \otimes X \right)
\]
is equal to
\[
  \pi \left(\overline{(\xi_i \otimes \eta \otimes \zeta \otimes \overline{\rho_U^{-1/2} \xi_j})} \otimes (\id_{UV} \otimes R_U^* \otimes_{W\bar{U}}) X \right),
\]
which is equal to $\pi \left( (\bar{\xi}_i \otimes \xi_j) \tilderhd (ab) \right)$ since $(R_U^* \otimes \id_U) (\id_U \otimes \bar{R}_U) = \id_U$.
\end{proof}

\begin{lemma}\label{ydstructurelemma4}
The $\cO(G)$-action on $\cB_\cM$ is compatible with the $*$-structure:
\[
  x \rhd a^* = \left( S(x)^* \rhd a\right)^*.
\]
\end{lemma}
\begin{proof}
For $x = \bar{\xi} \otimes \zeta \in \bar{H}_U \otimes H_U$, denote $\bar{\zeta} \otimes \xi$ by $x^\dagger$.
We then have $S(\pi_G(x))^* = \pi_G(x^\dagger)$, see~\cite{NY3}.
Thus, the claim is equivalent to
\[
  \pi( x \tilderhd a^\dagger) = \pi((x^\dagger \tilderhd a))^\dagger.
\]

The left hand side is
\[
  \pi \left(\overline{(\xi \otimes \overline{\rho_V^{-1/2} \eta} \otimes \overline{\rho_U^{-1/2}\zeta})} \otimes (\sigma_U \otimes \id_{V \bar{U}})(\id_U \otimes T^* \otimes \id_{\bar{V} \bar{U}})(\id_U \otimes \bar{R}_V \otimes \id_{\bar{U}})(\sigma_U^* \otimes \id_{\bar{U}}) \bar{R}_U \right) .
\]
The right hand side is
\begin{multline*}
  \pi \Bigl( \overline{ \overline{(\rho_U^{-1/2} \zeta \otimes \rho_V^{-1/2} \eta \otimes \bar{\xi})}}  \otimes (\bar{R}_U^* \otimes \id_{\overline{UV\bar{U}}})(\sigma_U \otimes \id_{\bar{U} \overline{UV\bar{U}}})(\id_U \otimes T^* \otimes \id_{\bar{U} \overline{UV\bar{U}}})\\
  (\sigma_U^* \otimes \id_{V \bar{U}\overline{UV\bar{U}}}) \bar{R}_{UV\bar{U}} \Bigr).
\end{multline*}
The unitary $K\colon H_U \otimes H_{\bar{V}} \otimes H_{\bar{U}} \to H_{\overline{UV \bar{U}}}$, $\xi \otimes \bar{\eta} \otimes \bar{\zeta} \mapsto \overline{\zeta \otimes \eta \otimes \bar{\xi}}$ is an equivalence of representations, and
\[
  \bar{R}_{UV\bar{U}} = (\id_{UV\bar{U}} \otimes K)(\id_{UV} \otimes R_U \otimes \id_{\bar{V}\bar{U}})(\id_U \otimes \bar{R}_V \otimes \id_{\bar{U}}) \bar{R}_U.
\]
With this and the identity $(\bar{R}_U^* \otimes \id_U) (\id_U \otimes R_U) = \id_U$, we see that the claim is true.
\end{proof}

\begin{lemma}\label{ydstructurelemma5}
The Yetter--Drinfeld condition for $\rhd$ is satisfied: if $\alpha$ is the action of $G$ on $B_\cM$, then
\[
  \alpha( x \rhd a) = x_{(1)} a_{(1)} S(x_{(3)}) \otimes (x_{(2)} \rhd a_{(2)})
\]
for all $ x \in \cO(G)$ and for all $a \in \cB_\cM$.
\end{lemma}

\begin{proof}
Fix orthonormal basis $\{\xi_i\}_i$ and $\{\eta_k\}_k$ of $H_U$ and $H_V$, respectively, consisting of eigenvectors of the corresponding operators $\rho_U$ and $\rho_V$.
Let $\rho_i$ be the associated eigenvalues of $\rho_U$, so that we have $\rho_U \xi_i = \rho_i \xi_i$.
Take $x = u_{i_0j_0} = \pi_G( \bar{\xi}_{i_0} \otimes \xi_{j_0} ) \in \cO(G)$ and $a = \pi(\bar{\eta}_{k_0} \otimes T) \in \pi( \bar{H}_V \otimes \cM(m, m \stmryolt V))$.
Recall that the $G$-action on $\cB_\cM$ is defined by
\[
  \alpha (\pi(\bar{\eta}_{k_0} \otimes T)) = \sum_k v_{k_0 k} \otimes \pi( \bar{\eta}_k \otimes T) ,
\]
with $v_{k_1 k_2}$ denoting the matrix coefficients of $V$ with respect to the basis $\{\eta_k\}_k$.
From the definitions, we get that $x_{(1)} a_{(1)} S(x_{(3)}) \otimes (x_{(2)} \rhd a_{(2)})$ equals to
\begin{equation*}
  \sum_{i,j,k}\rho_{j}^{-1/2} u_{i_0i}v_{k_0k} S(u_{jj_0}) \otimes \pi \left( \overline{(\xi_i \otimes \eta_k \otimes \bar{\xi}_j )} \otimes (\sigma_U \otimes \id_{V \bar{U}})(\id_U \otimes T \otimes \id_{\bar{U}})(\sigma_U^* \otimes \id_{\bar{U}}) \bar{R}_U  \right)
\end{equation*}
The contragredient representation of $U$ has $\{u_{ij}^*\}_{i,j}$ as matrix coefficients.
As for the dual representation $\bar{U}$, the matrix coefficients are $\{ \bar{u}_{ij} = \rho_i^{1/2}\rho_j^{-1/2} u_{ij}^* = \rho_i^{1/2}\rho_j^{-1/2} S(u_{ji})\}_{i,j}$.
Therefore, $\rho_j^{-1/2} S(u_{jj_0}) = \rho_{j_0}^{-1/2} \bar{u}_{j_0j}$, and the expression above is the same as
\[
  \rho_{j_0}^{-1/2} \sum_{i,j,k} u_{i_0i}v_{k_0k} \bar{u}_{jj_0} \otimes \pi \left( \overline{(\xi_i \otimes \eta_k \otimes \bar{\xi}_j )} \otimes (\sigma_U \otimes \id_{V \bar{U}})(\id_U \otimes T \otimes \id_{\bar{U}})(\sigma_U^* \otimes \id_{\bar{U}}) \bar{R}_U  \right).
\]
This is the right hand side of the equation in the Lemma; let us compute the left hand side.

As $u_{i_0j_0} \rhd (\bar{\eta}_{k_0} \otimes T) \in \pi ( \bar{H}_{UV\bar{U}} \otimes \cM(m, m \stmryolt (U \otimes V \otimes \bar{U})))$,
\[
  \alpha ( u_{i_0j_0} \rhd (\bar{\eta}_{k_0} \otimes T)) = (U\otimes V \otimes \bar{U})_{21}^*(1 \otimes (u_{i_0j_0} \rhd (\bar{\eta}_{k_0} \otimes T))).
\]
Expanding on the definitions, we see that it is the same as $x_{(1)} a_{(1)}
  S(x_{(3)}) \otimes (x_{(2)} \rhd a_{(2)})$.
\end{proof}

Lemmas~\ref{ydstructurelemma1}--\ref{ydstructurelemma5} say that the closure $B_\cM$ of $\cB_\cM$ becomes then a Yetter--Drinfeld $G$-$C^*$-algebra.

\subsection{Duality}

Our next goal is to prove that, up to isomorphisms, the above constructions are inverse to each other.
Given a centrally pointed bimodule category $\cM$, let us write $\cD_\cM = \cD_{B_\cM}$.
Then by the duality for $G$-C$^*$-algebras, there is an equivalence of right $\Rep(G)$-module $C^*$-categories, $\Psi\colon \cM \to \cD_\cM$~\cite{n1}.

\begin{proposition}\label{prop:bimod-YD-bimod}
The right module category equivalence functor $\Psi$ extends to a central functor.
\end{proposition}

\begin{proof}
By Proposition~\ref{prop:equiv-to-triv-braiding-cat}, it is enough to construct a corresponding extension from $\tilde\cM$ to $\tilde\cD_\cM$.

First we note that $\tilde\cD_\cM$ is the category $\cC_\cM = \cC_{B_\cM}$ in~\cite{NY3}, i.e., the idempotent completion of $\cC$ with respect to the morphism sets
\[
  \cC_{B_\cM}(U, V) = \Hom_{\tilde{\cD}_\cM}(H_U \otimes B_\cM, H_V \otimes B_\cM),
\]
the right-hand-side meaning $G$-equivariant Hilbert $B_\cM$-module maps.
Then, the right module category equivalence $\Psi\colon\cM \to \cC_{B_\cM}$ is characterized by $m \stmryolt U \mapsto U$ at the level of objects, and at the level of morphisms by
\begin{equation}\label{eq:equiv-from-mod-cat-to-mod-cat-of-alg}
  \Psi(T) = \sum_{i,j} \theta_{\zeta_j, \xi_i} \otimes \pi \left( \overline{(\zeta_j \otimes \overline{\rho^{-1/2} \xi_i})} \otimes (T \otimes \id_{\bar U}) (\id_m \stmryolt \bar{R}_U) \right)
\end{equation}
for $T \in \cM(m \stmryolt U, m \stmryolt V)$, where $(\xi_i)_i$ is an orthonormal basis of $H_U$, $(\zeta_j)_j$ is an orthonormal basis of $H_V$, and $\theta_{\zeta, \xi}$ is the linear map $\eta \mapsto (\eta, \xi) \zeta$, see~\cite{NY3}.

Now let us denote the central structure on $\cM = (\cM,m)$ by $\sigma$, and the one on $\cC_\cM$ by $S$.
Let us fix $T \in \cM(m_1 \stmryolt V,m_1 \stmryolt W)$, and take orthonormal bases $(\xi_i)_i$ on $H_U$, $(\zeta_l)_l$ on $H_V$, and $(\eta_k)_k$ on $H_W$.

On the one hand, the morphism $\Psi \left( (\sigma \otimes \id_W) (\id_U \otimes T) (\sigma^{-1} \otimes \id_V) \right)$ can be expanded as
\[
  \sum_{i,j,k,l} \theta_{\xi_i \otimes \eta_k,\xi_j \otimes \zeta_l} \otimes \pi \left( \overline{(\xi_i \otimes \eta_k \otimes \overline{\rho^{-1/2}(\xi_j \otimes \zeta_l)} )} \otimes \left( (\sigma \otimes \id_W) (\id_U \otimes T) (\sigma^{-1} \otimes \id_V) \otimes \id_{\overline{UV}} \right) \bar{R}_{UV} \right),
\]
which is equal to
\[
  \sum  m^U_{ij} \otimes \theta_{\eta_k,\zeta_l} \otimes \left( u_{ij} \rhd \pi \left( \overline{(\eta_k \otimes \overline{\rho^{-1/2 }\zeta_l}}) \otimes (T \otimes \id_{\bar{V}}) \bar{R}_V \right) \right).
\]
On the other, writing $\Psi(T) = \sum_l \theta_l \otimes b_l$ with $\theta_l \in B(H_V, H_W)$ and $b_l \in B_\cM$, we have
\[
  (S \otimes \id_W) \left(\id_U \otimes \sum_l T_l \otimes b_l\right) (S^{-1} \otimes \id_V) = \sum_{l,i,j} m^U_{ij} \otimes T_l \otimes (u_{ij} \rhd b_l).
\]
Comparing these we get
\[
  \Psi \left( (\sigma \otimes \id_W) (\id_U \otimes T) (\sigma^{-1} \otimes \id_V) \right) = (S \otimes \id_W) (\id_U \otimes \Psi(T)) (S^{-1} \otimes \id_V),
\]
which gives the desired extension by Proposition~\ref{prop:strictify-mod-ftr}.
\end{proof}

Going the other way, let $B$ be a Yetter--Drinfeld $G$-$C^*$-algebra, and consider the new algebra $B_{\cD_B}$ we get by the above procedure.
Let us denote their regular subalgebras by $\cB$ and $\cB_\cD$.
Then the isomorphism of $G$-algebras $\lambda\colon B_{\cD_B} \to B$ is given by
\[
  \lambda (\bar{\eta} \otimes T) = (\eta^* \otimes \id_B) (T)
\]
for $\bar\eta \in \bar H_V$ and $T \in (H_V \otimes B)^G \simeq \cD_B(B, B \stmryolt V)$, where $\eta^*$ is the functional $\xi \mapsto (\xi, \eta)$ on $H_V$.

\begin{proposition}\label{yetterdrinfeldprop85}
$\lambda$ is $\hat{G}$-equivariant.
\end{proposition}

\begin{proof}
Fix an orthonormal basis $(\zeta_k)_k$ in $H_V$ and take $T \in \cD_B(B, B \stmryolt V)$, represented by $\sum_k \zeta_k \otimes b_k \in H_V \otimes B$ up to the above correspondence.

The central structure $S$ on $\cD_B$ satisfies
\begin{equation*}
  (S \otimes \id_{V\bar{U}}) (\id_U \otimes T \otimes \id_{\bar{U}}) (S^{-1} \otimes \id_{\bar{U}}) = \sum_{i, j, k} m^U_{ij} \otimes \zeta_k \otimes 1 \otimes (u_{ij} \rhd b_k)
\end{equation*}
in $\cD_B(B \stmryolt U \otimes \bar{U}, B \stmryolt U \otimes V \otimes \bar{U})$.

From our definition of the action $\rhd$, we see that $u_{i_0j_0} \rhd \pi( \bar{\zeta}_{k_0} \otimes T)$ is equal to
\begin{multline*}
  \pi \left( (\bar{\xi}_{i_0} \otimes \xi_{j_0}) \tilderhd (\bar{\zeta}_{k_0} \otimes T) \right) \\
  = \pi \left( \overline{(\xi_{i_0} \otimes \zeta_{k_0} \otimes \overline{\rho_U^{-1/2} \xi_{j_0}})} \otimes (S \otimes \id_{V\bar{U}}) (\id_U \otimes T \otimes \id_{\bar{U}}) (S^{-1} \otimes \id_{\bar{U}}) \bar{R}_U  \right) \\
  = \pi \left( \overline{(\xi_{i_0} \otimes \zeta_{k_0} \otimes \overline{\rho_U^{-1/2} \xi_{j_0}})} \otimes  \left( \sum_{i, j, k} m^U_{ij} \otimes \zeta_k \otimes 1 \otimes (u_{ij} \rhd b_k) \right) \bar{R}_U  \right).
\end{multline*}
From~\eqref{eq:duality-mors}, this is equal to
\[
  \sum_{i,j,k} \overline{(\xi_{i_0} \otimes \zeta_{k_0} \otimes \overline{\rho_U^{-1/2} \xi_{j_0}})} \otimes ( \xi_i \otimes \zeta_k \otimes \overline{\rho_U^{1/2} \xi_j }) \otimes (u_{ij} \rhd b_k).
\]
We thus obtain
\[
  \lambda \left(u_{i_0j_0} \rhd \pi (\bar{\zeta}_{k_0} \otimes T) \right) = u_{i_0j_0} \rhd b_{k_0} = u_{i_0, j_0} \rhd \lambda(\bar \zeta_{k_0} \otimes T),
\]
establishing the claim.
\end{proof}

\subsection{Moduli of module functors}

Let $\cM_1$ and $\cM_2$ be two centrally pointed $\Rep(G)$-bimodule categories, with respective generators $m_1$ and $m_2$ and half-braidings $\sigma$ and $\sigma'$, and denote their corresponding Yetter--Drinfeld $G$-$C^*$-algebras by $B_1$ and $B_2$, respectively.
Consider a right $\Rep(G)$-module functor $F\colon \cM_1 \to \cM_2$ such that $F(m_1) = m_2$.
Let us work out the condition on $F$ such that the induced $G$-equivariant homomorphism $f\colon B_1 \to B_2$ becomes a $D(G)$-equivariant homomorphism.

Restricting to the regular subalgebras, $f$ is given by the collection of
\[
  \bar{H}_U \otimes \cM_1(m_1, m_1 \stmryolt U) \to \bar{H}_U \otimes \cM_2(m_2, m_2 \stmryolt U), \quad T \mapsto \id_{\bar{H}_U} \otimes F_2^{-1} F(T)
\]
for the irreducible representations $U$.
In view of~\eqref{eq:reg-alg-action-in-YD-str}, $f$ will be $\hat{G}$-equivariant if and only if
\begin{multline*}
  (F_{2; m,U V \bar U})^{-1} F \mathopen{}\left( (\sigma_U \stmryolt \id_V \otimes \id_{\bar{U}}) (\id_U \stmryogt T \stmryolt \id_{\bar{U}} ) (\sigma_U^{-1} \stmryolt \id_{\bar{U}} ) (\id_{m_1} \stmryolt \bar{R}_U) \right) \\
  = (\sigma_U' \stmryolt \id_V \otimes \id_{\bar{U}}) (\id_U \stmryogt F_{2;m,V}^{-1}F(T) \stmryolt \id_{\bar{U}} ) (\sigma_U'^{-1} \stmryolt \id_{\bar{U}} ) (\id_{m_2} \stmryolt \bar{R}_U)
\end{multline*}

As before, consider the braiding maps $\Sigma_{U;V} \colon \cM_1(m_1, m_1 \stmryolt V) \to \cM_1(m_1 \stmryolt U, m_1 \stmryolt U \stmryolt V)$ given by $\Sigma_{U; 1, V}$ in~\eqref{eq:br-map}.
We will write the analogous braiding maps on $\cM_2$ as $\Sigma'$.
Using
\[
  \Sigma_{U;V}(T) = (\id_{m_1 \stmryolt U \stmryolt V} \stmryolt R_U^*) (\sigma_U \stmryolt \id_{V \bar{U} U})(\id_U \stmryogt T \stmryolt \id_{U \bar{U}}) (\sigma_U^{-1} \stmryolt \id_{\bar{U}U})(\id_{m_1} \stmryolt \bar{R}_U \otimes \id_U),
\]
we see that $f$ will be $\hat{G}$-equivariant if and only if
\begin{equation}\label{eq:F-ad-sigma}
  F_{2;m_1,UV}^{-1} F(\Sigma_{U;V}(T)) F_2 = \Sigma_{U;V}'(F_{2;m,V}^{-1}F(T))
\end{equation}
for all $U,V \in \Rep(G)$ and $T \in \cM_1(m_1, m_1 \stmryolt V)$.

\begin{proposition}\label{yetterdrinfeldprop6}
Let $F\colon \cM_1 \to \cM_2$ be a right C$^*$-$\Rep(G)$-module functor between centrally pointed bimodule categories.
If $F$ satisfies~\eqref{eq:F-ad-sigma}, then there is a strict bimodule functor $\tilde F \colon \tilde\cM_1 \to \tilde\cM_2$ between the categories obtained by trivialization of the central structures, that is naturally isomorphic to $F$ up to the equivalence of Proposition~\ref{prop:equiv-to-triv-braiding-cat}.
\end{proposition}

\begin{proof}
From the assumption we actually have
\[
  F_{2;m_1,UW}^{-1} F(\Sigma_{U;V}(T)) F_{2;m_1,UV} = \Sigma_{U;V}'(F_{2;m,V}^{-1}F(T)) \quad (T \in \cM_1(m_1 \stmryolt V, m_1 \stmryolt W)).
\]
Indeed, given $T \in \cM_1(m_1 \stmryolt V, m_1 \stmryolt W)$, we can take $S = (T \stmryolt \id_{\bar W}) (\id_{m_1} \stmryolt \bar R_W)$ and combine~\eqref{eq:F-ad-sigma} with the conjugate equation to get this claim.
Then the claim follows from Proposition~\ref{prop:strictify-mod-ftr}.
\end{proof}

Finally, let us compare different choices of bimodule functors leading to the same homomorphisms.
Suppose that $F, F' \colon \cM \to \cM'$ are central functors that are naturally isomorphic.
Then the induced homomorphisms $B_\cM \to B_{\cM'}$ agree.
Indeed, we just need to look at the induced maps
\[
  \cM(m, m \stmryolt U) \to \cM'(m', m' \stmryolt U)
\]
which are given by $T \mapsto (F_0^{-1} \otimes \id_U) F_2^{-1} F(T) F_0$ and a similar formula for $F'$.
Using the commutativity of diagrams in Definition~\ref{def:nat-trans-bimod-functors}, we see that these maps indeed agree.

\begin{proposition}
Suppose that the induced homomorphisms $B_\cM \to B_{\cM'}$ agree.
Then $F$ and $F'$ are naturally isomorphic as bimodule functors.
\end{proposition}

\begin{proof}
Let  $f_F \colon B_\cM \to B_{\cM'}$ denote the homomorphism induced by $F$.
By Proposition~\ref{prop:bimod-YD-bimod}, the assertion follows once we can check that the bimodule functor $F_f \colon \cD_\cM \to \cD_{\cM'}$ induced by $f_F$ is naturally isomorphic to $F$ up to the bimodule equivalence $\cM \to \cD_\cM$.
Looking at the corresponding functor $\tilde \cM \to \tilde \cM'$ induced by $F$, the claim is tautological.
Then we get the claim by Proposition~\ref{prop:equiv-to-triv-braiding-cat}.
\end{proof}

\raggedright


\begin{bibdiv}
\begin{biblist}

\bib{AV1}{misc}{
      author={Antoun, Jamie},
      author={Voigt, Christian},
       title={On bicolimits of {C$^*$}-categories},
         how={preprint},
        date={2020},
      eprint={\href{http://arxiv.org/abs/2006.06232}{\texttt{arXiv:2006.06232
  [math.OA]}}},
}


\bib{BZBJ2018}{article}{
   author={Ben-Zvi, David},
   author={Brochier, Adrien},
   author={Jordan, David},
   title={Integrating quantum groups over surfaces},
   journal={J. Topol.},
   volume={11},
   date={2018},
   number={4},
   pages={874--917},
   issn={1753-8416},
   review={\MR{3847209}},
   doi={10.1112/topo.12072},
}

\bib{ydk1}{article}{
      author={De~Commer, Kenny},
      author={Yamashita, Makoto},
       title={Tannaka-{K}re\u\i n duality for compact quantum homogeneous
  spaces. {I}. {G}eneral theory},
        date={2013},
        ISSN={1201-561X},
     journal={Theory Appl. Categ.},
      volume={28},
       pages={No. 31, 1099\ndash 1138},
      eprint={\href{http://arxiv.org/abs/1211.6552}{\texttt{arXiv:1211.6552
  [math.OA]}}},
      review={\MR{3121622}},
}

\bib{FreslonTaipeWang}{article}{
   author={Freslon, Amaury},
   author={Taipe, Frank},
   author={Wang, Simeng},
   title={Tannaka--Krein reconstruction and ergodic actions of easy quantum
   groups},
   journal={Comm. Math. Phys.},
   volume={399},
   date={2023},
   number={1},
   pages={105--172},
   issn={0010-3616},
   review={\MR{4567370}},
   doi={10.1007/s00220-022-04555-y},
}

\bib{HHN1}{article}{
   author={Habbestad, Erik},
   author={Hataishi, Lucas},
   author={Neshveyev, Sergey},
   title={Noncommutative Poisson boundaries and Furstenberg-Hamana
   boundaries of Drinfeld doubles},
   language={English, with English and French summaries},
   journal={J. Math. Pures Appl. (9)},
   volume={159},
   date={2022},
   pages={313--347},
   issn={0021-7824},
   review={\MR{4377998}},
   doi={10.1016/j.matpur.2021.12.006},
}

\bib{HY2022}{article}{
  title={Injectivity for algebras and categories with quantum symmetry},
  author={Hataishi, Lucas}, 
  author= {Yamashita, Makoto},
  journal={arXiv preprint arXiv:2205.06663},
  year={2022}
}

\bib{n1}{article}{
      author={Neshveyev, Sergey},
       title={Duality theory for nonergodic actions},
        date={2014},
        ISSN={1867-5778},
     journal={M{\"u}nster J. Math.},
      volume={7},
      number={2},
       pages={413\ndash 437},
      eprint={\href{http://arxiv.org/abs/1303.6207}{\texttt{arXiv:1303.6207
  [math.OA]}}},
      review={\MR{3426224}},
}

\bib{MR3204665}{book}{
      author={Neshveyev, Sergey},
      author={Tuset, Lars},
       title={Compact quantum groups and their representation categories},
      series={Cours Sp{\'e}cialis{\'e}s [Specialized Courses]},
   publisher={Soci{\'e}t{\'e} Math{\'e}matique de France, Paris},
        date={2013},
      volume={20},
        ISBN={978-2-85629-777-3},
      review={\MR{3204665}},
}

\bib{NY3}{article}{
      author={Neshveyev, Sergey},
      author={Yamashita, Makoto},
       title={Categorical duality for {Y}etter-{D}rinfeld algebras},
        date={2014},
        ISSN={1431-0635},
     journal={Doc. Math.},
      volume={19},
       pages={1105\ndash 1139},
      eprint={\href{http://arxiv.org/abs/1310.4407}{\texttt{arXiv:1310.4407
  [math.OA]}}},
      review={\MR{3291643}},
}

\bib{NY2}{article}{
  author={Neshveyev, Sergey},
   author={Yamashita, Makoto},
  title={Drinfeld center and representation theory for monoidal categories},
  journal={Communications in Mathematical Physics},
  volume={345},
  number={1},
  pages={385--434},
  year={2016},
  publisher={Springer}
}

\bib{NY1}{article}{
   author={Neshveyev, Sergey},
   author={Yamashita, Makoto},
   title={Poisson boundaries of monoidal categories},
   journal={Ann. Sci. \'{E}c. Norm. Sup\'{e}r. (4)},
   volume={50},
   date={2017},
   number={4},
   pages={927--972},
   issn={0012-9593},
   review={\MR{3679617}},
   doi={10.24033/asens.2335},
}

\bib{MR3933035}{article}{
      author={Neshveyev, Sergey},
      author={Yamashita, Makoto},
       title={Categorically {M}orita equivalent compact quantum groups},
        date={2018},
        ISSN={1431-0635},
     journal={Doc. Math.},
      volume={23},
       pages={2165\ndash 2216},
      review={\MR{3933035}},
}

\bib{MR2566309}{article}{
      author={Nest, Ryszard},
      author={Voigt, Christian},
       title={Equivariant {P}oincar\'e duality for quantum group actions},
        date={2010},
        ISSN={0022-1236},
     journal={J. Funct. Anal.},
      volume={258},
      number={5},
       pages={1466\ndash 1503},
      eprint={\href{http://arxiv.org/abs/0902.3987}{\texttt{arXiv:0902.3987
  [math.KT]}}},
         url={http://dx.doi.org/10.1016/j.jfa.2009.10.015},
         doi={10.1016/j.jfa.2009.10.015},
      review={\MR{2566309}},
}

\bib{ostrik03}{article}{
  title={Module categories, weak Hopf algebras and modular invariants},
  author={Ostrik, Victor},
  journal={Transformation groups},
  volume={8},
  number={2},
  pages={177--206},
  year={2003},
  publisher={Springer}
}

\bib{pr1}{article}{
title={A duality theorem for ergodic actions of compact quantum groups on C*-algebras},
  author={Pinzari, Claudia},
  author={Roberts, John E},
  journal={Communications in mathematical physics},
  volume={277},
  number={2},
  pages={385--421},
  year={2008},
  publisher={Springer}
}

\bib{MR943923}{article}{
      author={Woronowicz, S.~L.},
       title={Tannaka-{K}re\u\i n duality for compact matrix pseudogroups.
  {T}wisted {${\rm SU}(N)$} groups},
        date={1988},
        ISSN={0020-9910},
     journal={Invent. Math.},
      volume={93},
      number={1},
       pages={35\ndash 76},
         url={http://dx.doi.org/10.1007/BF01393687},
         doi={10.1007/BF01393687},
      review={\MR{943923 (90e:22033)}},
}

\end{biblist}
\end{bibdiv}
\end{document}